\documentclass[12pt, a4paper]{article}
\usepackage{mathtools,amsthm,amsmath,amsfonts,amssymb,tikz-cd,enumitem, graphicx,mathrsfs,bbm,stmaryrd,xcolor}

\tikzcdset{scale cd/.style={every label/.append style={scale=#1},
    cells={nodes={scale=#1}}}}
\usepackage{authblk}
\usepackage{url}
\usepackage{esvect}
\makeatletter
\newcommand{\address}[1]{\gdef\@address{#1}}
\newcommand{\email}[1]{\gdef\@email{\url{#1}}}
\newcommand{\@endstuff}{\par\vspace{\baselineskip}\noindent\small
\begin{tabular}{@{}l}\@address\\\textit{E-mail address:} \@email\end{tabular}}
\AtEndDocument{\@endstuff}
\makeatother
\usepackage{setspace}
\usepackage[utf8]{inputenc}
\usepackage{CJKutf8}
\counterwithin{equation}{section}
\usepackage[pdfencoding=auto,psdextra,colorlinks=true,linkcolor=blue,citecolor=blue,urlcolor = orange]{hyperref}
\usepackage[margin = 1.2in]{geometry}
\usepackage{kantlipsum}
\allowdisplaybreaks
\newtheorem{theorem}{Theorem}[section]
\newtheorem{definition}[theorem]{Definition}

\newtheorem{remark}[theorem]{Remark}
\newtheorem{proposition}[theorem]{Proposition}
\newtheorem{corollary}[theorem]{Corollary}
\newtheorem{assumption}[theorem]{Assumption}
\newcommand{\citep}[1]{\cite{#1}}

\newcommand{\professor}{\text{Prof.\ }\hspace{-0.03125mm}}
\newcommand{\doctor}{\text{Dr.\ }\hspace{-0.03125mm}}
\newcommand{\edmp}{\text{\normalfont End}}

\begin{document}
\title{\textbf{Mapping cone Thom forms}}
\author{Hao Zhuang}
\address{Beijing International Center for Mathematical Research, Peking University}
\email{hzhuang@pku.edu.cn}
\date{\today}
\maketitle
\begin{abstract}
For the de Rham mapping cone cochain complex induced by a smooth closed $2$-form, we explicitly write down the associated mapping cone Thom form in the sense of Mathai-Quillen. Our construction uses the mapping cone covariant derivative, carrying the extra information brought by the $2$-form. Our main tool is the Berezin integral. As the main result, we show that this Thom form is closed with respect to the mapping cone differentiation, its integration along the fiber is $1$, and it satisfies the transgression formula. 

\end{abstract} 
\tableofcontents
\section{Introduction}
The de Rham mapping cone cochain complex associated with a smooth closed form carries the studies from various perspectives in geometry and topology. The symplectic case connects with the primitive cohomology and the filtered cohomology \cite{tanaka_tseng_2018, tty3rd, tty1st, tty2nd}. For the de Rham mapping cone cochain complex associated with a general closed form, we have seen the studies in characteristic classes \cite{symplectic_semi_char_2025, transgression_primitive_2025, 1_filtered_semi_char_2025}, gauge theory \cite{tseng_and_zhou_symplectic_flat_connection_and_twisted_primitive2022, tseng_zhou_symplectic_flat_functional_characteristic_classes2022,  tseng_and_zhou_2025mapping_yang_mills}, Morse theory \cite{tangtsengclausenmappingcone, tangtsengclausensymplecticwitten, 2026_instanton_construction_mapping_cone}, and so on. 

An important mechanism supporting the studies of characteristic classes, gauge theory, and Morse theory is the Thom isomorphism \cite[Section 6]{bott2013differential}. The analytic expression of the associated Thom form is explicitly formulated by Mathai and Quillen \cite[Theorem 4.5]{Mathai-Quillen_Thom_form}. Now, in the mapping cone case, we also expect to build such a mechanism. In \cite[Theorem 3.13]{morrison2024cohomology}, the Thom isomorphism is formulated via the topological approach. This reminds us of the necessity to give an explicit construction of the mapping cone Thom form in the sense of Mathai-Quillen. 

In this paper, we give the construction. We assume the following settings: 
\begin{assumption}\normalfont
    We let $M$ be a smooth manifold without boundary, $\omega$ be a smooth closed $2$-form on $M$, and $E$ be an oriented smooth vector bundle of rank $n$ over $M$. Also, we equip $E$ with a smooth vector bundle metric $g$. 
\end{assumption}

For the construction of the mapping cone Thom form, we follow the pattern in \cite[Section 1.6]{bgv}, \cite[Chapter 3]{wittendeformationweipingzhang}, and \cite[Chapter 3]{weipingzhangnewedition}. For the notion of the mapping cone covariant derivative, we follow \cite[Proposition 3.8]{tseng_and_zhou_symplectic_flat_connection_and_twisted_primitive2022} and \cite[Section 2.1]{tseng_and_zhou_2025mapping_yang_mills}.

With respect to $g$, we let $\nabla$ be a Euclidean connection on $E$, and let $\Phi\in\Omega^0(M,\edmp(E))$ be skew-adjoint. 
They form a mapping cone covariant derivative
\begin{align}
    \mathbb{A}: \Omega^i(M,E)\oplus\Omega^{i-1}(M,E)&\to \Omega^{i+1}(M,E)\oplus\Omega^{i}(M,E)   \nonumber \\
       (\alpha, \beta) &\mapsto (\nabla\alpha+\omega\wedge\beta, \Phi\alpha -\nabla\beta).
\end{align}
Let $\sigma: E\to M$ be the projection to the base space. We then obtain the pullback bundle $\widetilde{E} = \sigma^*E$ and the associated
\begin{align}
    \widetilde{\omega} = \sigma^*\omega,\ 
    \widetilde{\nabla} = \sigma^*\nabla,\ 
    \widetilde{\Phi} = \sigma^*\Phi,\ 
    \widetilde{g} = \sigma^*g. 
\end{align}
Then, we have the de Rham mapping cone cochain complex
\begin{align}
    d^{\widetilde{\omega}}: \Omega^i(E)\oplus\Omega^{i-1}(E)&\to\Omega^{i+1}(E)\oplus\Omega^{i}(E)\nonumber\\
    (\alpha,\beta)&\mapsto (d\alpha + \widetilde{\omega}\wedge\beta, -d\beta)
\end{align}
of $(E,\widetilde{\omega})$, and the mapping cone covariant derivative 
\begin{align}
    \widetilde{\mathbb{A}}:  \Omega^i(E, \Lambda^*\widetilde{E})\oplus\Omega^{i-1}(E, \Lambda^*\widetilde{E})&\to\Omega^{i+1}(E, \Lambda^*\widetilde{E})\oplus\Omega^{i}(E, \Lambda^*\widetilde{E})   \nonumber\\
    (\alpha, \beta) &\mapsto \left(\widetilde{\nabla}\alpha + \widetilde{\omega}\wedge\beta,\ \widetilde{\Phi}^\Lambda\alpha-\widetilde{\nabla}\beta\right)
\end{align}
on $\Lambda^*E$. Here, $\widetilde{\Phi}^\Lambda$ is the derivation (See (\ref{derivation})-(\ref{derivation common pattern})) induced by $\widetilde{\Phi}$, and $\widetilde{\nabla}$ extends naturally to $\Lambda^*\widetilde{E}$. 

Let $e_1, \cdots, e_n$ be an oriented local orthonormal frame of $E$, then we have its lift $\widetilde{e}_1, \cdots, \widetilde{e}_n$ to $\widetilde{E}$. Then, we let 
\begin{align}
    Q_{\widetilde{\mathbb{A}}} =\ & \sum_{1\leqslant i<j\leqslant n}\widetilde{g}\left(\widetilde{\nabla}^2\widetilde{e}_i + \widetilde{\omega}\wedge\widetilde{\Phi}\widetilde{e}_i,\ \widetilde{e}_j\right) \otimes(\widetilde{e}_i\wedge \widetilde{e}_j)
    \in\Omega^2(E,\Lambda^2\widetilde{E}), 
\end{align}
and 
\begin{align}
    S_{\widetilde{\mathbb{A}}} =\ & \sum_{1\leqslant i< j\leqslant n}\widetilde{g}\left(\widetilde{\Phi}\widetilde{\nabla}\widetilde{e}_i - \widetilde{\nabla}\widetilde{\Phi}\widetilde{e}_i,\ \widetilde{e}_j\right)\otimes(\widetilde{e}_i\wedge \widetilde{e}_j) \in\Omega^1(E,\Lambda^2\widetilde{E}). 
\end{align}
They are globally defined because $\nabla$ is Euclidean and $\Phi$ is skew-adjoint. 

We let $\mathbf{v}\in\Omega^0(E,\widetilde{E})$ be the tautological section (See (\ref{coordinate expression of tautological section})). Then, we define
\begin{align}
    \mathcal{A} = \left(\dfrac{1}{2}|\mathbf{v}|^2, 0\right) + \widetilde{\mathbb{A}}(\mathbf{v},0) - \left(Q_{\widetilde{\mathbb{A}}}, S_{\widetilde{\mathbb{A}}}\right).
\end{align}
By the Taylor series of $f(z) = e^z$ and the wedge product on pairs of forms, we let
\begin{align}
    \mathcal{U} = (-1)^{n(n+1)/2}\left(\dfrac{1}{2\pi}\right)^{n/2}\int^B e^{-\mathcal{A}}. 
\end{align}
Here, $\displaystyle\int^B$ means applying the Berezin integral to each form in a pair. 

By verifying the closedness and the integration along the fiber, our first main result shows that $\mathcal{U}$ is the mapping cone Thom form that we want:  
\begin{theorem}\label{main result}
    The pair $\mathcal{U}\in\Omega^n(E)\oplus\Omega^{n-1}(E)$ is $d^{\widetilde{\omega}}$-closed. Also, it satisfies
    \begin{align}
        \int_{E/M}\mathcal{U} = (1,0).
    \end{align}
    Here, $\displaystyle\int_{E/M}$ means the integration along the fiber.
\end{theorem}

Let $\nabla_t$ be a smooth family of Euclidean connections, and $\Phi_t$ be a smooth family of skew-adjoint endomorphisms ($t\in\mathbb{R}$). We have the transgression: 
\begin{theorem}\label{main result 2}
    Let $\omega$ and $g$ be fixed. For the family of mapping cone Thom forms $\mathcal{U}_t$ associated with $\nabla_t$ and $\Phi_t$ $(t\in\mathbb{R})$, we have $\psi_t, \rho_t\in\Omega^*(E)$ such that 
    \begin{align}
        \dfrac{d\hspace{+0.5mm}\mathcal{U}_t}{dt} = d^{\widetilde{\omega}}(\psi_t, \rho_t).
    \end{align}
    The precise expression of $(\psi_t, \rho_t)$ is given in Proposition \ref{proposition of type 1 transgression}.
\end{theorem}
Comparing with \cite[Section 1.6]{bgv}, \cite[Chapter 3]{wittendeformationweipingzhang}, \cite[Chapter 3]{weipingzhangnewedition}, in the verification of Theorem \ref{main result} and Theorem \ref{main result 2}, the extra terms brought by $\omega$ and $\Phi$ cause the main difficulties. We need $\Phi$ to be skew-adjoint to approach extra terms. This make $\mathbb{A}$ an analogue of the Euclidean connection.

The paper is organized in this order: In Section \ref{section of mappning cone covariant derivative}, we review the mapping cone covariant derivative and its extension to exterior bundle valued forms. In Section \ref{section of bianchi identity}, we prove a skew-adjoint version of the mapping cone Bianchi identity, which is needed for Theorem \ref{main result}. In Section \ref{section of berezin integral}, we recall the Berezin integral and extend it to the mapping cone situation. In Section \ref{section of mapping cone thom}, we use the extended Berezin integral to construct $\mathcal{U}$ and prove Theorem \ref{main result}. In Section \ref{section of transgression}, we prove Theorem \ref{main result 2} and present some inspirations related to this study. 

\section*{Acknowledgments}
I thank \professor Xiaobo Liu, \professor Xiang Tang, \professor Li-Sheng Tseng, \professor Shanwen Wang, \doctor Jinxuan Chen, \doctor David Clausen, \doctor Danhua Song, and \doctor Shitan Xu for helpful discussions. Also, I thank Beijing International Center for Mathematical Research for providing an excellent working environment.

\section{Mapping cone covariant derivative}\label{section of mappning cone covariant derivative}
In this section, we review the mapping cone covariant derivative on vector bundle valued differential forms. At present, we do not require $\nabla$ to be Euclidean or $\Phi$ to be skew-adjoint. 

Let $E$ be a smooth vector bundle over $M$. Given a Koszul connection $\nabla$ on $E$, it induces the covariant derivative \cite[(1.2)]{wittendeformationweipingzhang}
\begin{align}
    \nabla: \Omega^i(M,E)&\to\Omega^{i+1}(M,E)\ \ (0\leqslant i\leqslant \dim M)
\end{align}
on $E$-valued forms. This $\nabla$ is determined by 
\begin{align}
   \nabla(\eta\otimes v)  = d\eta \otimes v + (-1)^{i}\eta\wedge\nabla v
\end{align}
for any $\eta\in\Omega^i(M)$ and any smooth section $v$ of $E$.

We let 
\begin{align}
    \Lambda^*E \coloneqq \bigoplus_{k = 0}^n\Lambda^k E
\end{align}
be the exterior bundle of $E$. Then, the connection $\nabla$ extends to $\Lambda^*E$: 
\begin{align}
    \nabla: \Gamma(TM)\times\Gamma(\Lambda^*E) &\to \Gamma(\Lambda^*E).
\end{align}
Here, $\Gamma(\Lambda^*E)$ (resp. $\Gamma(TM)$) is the space of smooth sections of $\Lambda^*E$ (resp. $TM$). Locally, when we have a smooth vector field $X$ on $M$, a smooth function $f$ on $M$, and a local frame
    $e_1, e_2, \cdots, e_n$
of $E$, then for $0\leqslant i_1<\cdots<i_r\leqslant n$, 
\begin{align}
    & \nabla_{X}(f e_{i_1}\wedge\cdots\wedge e_{i_r})  \nonumber \\
    =\ & X(f) e_{i_1}\wedge\cdots\wedge e_{i_r} + f\sum_{j = 1}^r e_{i_1}\wedge\cdots\wedge\nabla_X e_{i_j}\wedge\cdots\wedge e_{i_r}. 
\end{align}
Immediately, we obtain the associated covariant derivative
\begin{align}\label{derivative on exterior bundle}
    \nabla: \Omega^i(M,\Lambda^*E)&\to\Omega^{i+1}(M,\Lambda^*E) \ \ (0\leqslant i\leqslant \dim M)
\end{align}
determined by 
\begin{align}
    & \nabla(\alpha\otimes (e_{i_1}\wedge\cdots\wedge e_{i_r})) \nonumber \\
    =\ & d\alpha\otimes (e_{i_1}\wedge\cdots\wedge e_{i_r}) + (-1)^{i}\alpha\wedge\nabla(e_{i_1}\wedge\cdots\wedge e_{i_r})
\end{align}
for all $\alpha\in\Omega^i(M)$. 

Now, we generalize the above machinery for studying the de Rham mapping cone cochain complex. We start with the mapping cone covariant derivative proposed in \cite[Proposition 3.8]{tseng_and_zhou_symplectic_flat_connection_and_twisted_primitive2022} and \cite[Section 2.1]{tseng_and_zhou_2025mapping_yang_mills}. For the mapping cone superconnection defined on a superbundle, see \cite[Definition 1.3]{transgression_primitive_2025}. 
\begin{definition}\label{extended superconnection}
    \normalfont Suppose that $\Phi\in\Omega^0(M,\edmp(E))$. We call the linear map
\begin{align}
    \mathbb{A}: \Omega^i(M,E)\oplus\Omega^{i-1}(M,E)&\to \Omega^{i+1}(M,E)\oplus\Omega^{i}(M,E)   \nonumber \\
       (\alpha, \beta) &\mapsto (\nabla\alpha+\omega\wedge\beta, \Phi\alpha -\nabla\beta)
\end{align}
a mapping cone covariant derivative.  
\end{definition}

The $\Phi$ preserves the $C^\infty(M)$-module structure on the space of smooth sections of $E$. By \cite[Chapter III, \S 10.9, Proposition 14 \& (36)]{bourbaki_algebra_1}, it induces a derivation 
\begin{align}\label{derivation}
    \Phi^\Lambda\in\Omega^0(M,\edmp\left(\Lambda^*E\right)).
\end{align}
This derivation is defined as follows. Suppose that 
    $e_1, \cdots, e_n$
is a local frame of $E$. Then, for any 
\begin{align}
   a = \sum_{r = 0}^n\sum_{1\leqslant i_1<\cdots<i_r\leqslant n} a_{i_1\cdots i_r}e_{i_1}\wedge\cdots\wedge e_{i_r}
\end{align}
where $a_{i_1\cdots i_r}$'s are locally defined smooth functions, we have 
\begin{align}\label{derivation common pattern}
    \Phi^\Lambda a = \sum_{r = 1}^n\sum_{1\leqslant i_1<\cdots<i_r\leqslant n} \sum_{j = 1}^r a_{i_1\cdots i_r}e_{i_1}\wedge\cdots\wedge\Phi(e_{i_j})\wedge\cdots\wedge e_{i_r}.
\end{align}
The number $r$ in (\ref{derivation common pattern}) starts from $1$ because as a derivation compatible with the $C^\infty(M)$-module structure, $\Phi^\Lambda$ must be $0$ on $\Omega^0(M,\Lambda^0E) = C^\infty(M)$. 

With (\ref{derivative on exterior bundle}) and (\ref{derivation}), we have the following extension of $\mathbb{A}$: 
\begin{align}
    \mathbb{A}: \Omega^i(M,\Lambda^*E)\oplus\Omega^{i-1}(M,\Lambda^*E)&\to \Omega^{i+1}(M,\Lambda^*E)\oplus\Omega^{i}(M,\Lambda^*E)  \nonumber\\
    (\alpha, \beta)&\mapsto (\nabla\alpha + \omega\wedge\beta, \Phi^\Lambda\alpha - \nabla\beta).
\end{align}
Then, we define the analogue of the classical Euclidean connection: 
\begin{definition}\normalfont
    If with respect to some smooth vector bundle metric on $E$, the connection $\nabla$ is Euclidean, and the $\Phi$ is skew-adjoint, then we call $\mathbb{A}$ a Euclidean mapping cone covariant derivative under this metric.
\end{definition}

\section{Skew-adjoint Bianchi identity}\label{section of bianchi identity}
In this section, we prove the skew-adjoint Bianchi identity of $\mathbb{A}$. 
We assume that $E$ is an oriented smooth vector bundle of rank $n$ over $M$. Also, we equip $E$ with the smooth vector bundle metric $g$ and let $\mathbb{A}$ be Euclidean under $g$, i.e., $\nabla$ is Euclidean and $\Phi$ is skew-adjoint with respect to $g$.

Let $e_1,\cdots, e_n$ be an oriented local orthonormal frame of $E$. Then, we define 
\begin{align}\label{definition of q}
    Q_{\mathbb{A}} = \sum_{1\leqslant i<j\leqslant n} g\left({\nabla}^2{e}_i + {\omega}\wedge{\Phi}{e}_i,\ {e}_j\right)\otimes({e}_i\wedge {e}_j)
    \in\Omega^2(M,\Lambda^2{E})
\end{align}
and
\begin{align}\label{definition of s}
    S_{\mathbb{A}} = \sum_{1\leqslant i< j\leqslant n} g\left({\Phi}{\nabla}{e}_i - {\nabla}{\Phi}{e}_i,\ {e}_j\right)\otimes({e}_i\wedge{e}_j)\in\Omega^1(M,\Lambda^2{E}).
\end{align}
The $Q_{\mathbb{A}}$ and $S_{\mathbb{A}}$ are globally defined since $\mathbb{A}$ is Euclidean (cf. \cite[(3.9)]{wittendeformationweipingzhang}). 

We have the following skew-adjoint Bianchi identity. 
\begin{proposition}\label{Bianchi identity another expression}
    When $\mathbb{A}$ is Euclidean under $g$, we have  
    \begin{align}\label{a acts on q and s}
        {\mathbb{A}}\left(Q_{{\mathbb{A}}}, S_{{\mathbb{A}}}\right) = (0,0).
    \end{align}
\end{proposition}
\begin{proof}
Temporarily in this proof, we use the Einstein summation convention. Since 
    \begin{align}\label{curvature with connection}
          \mathbb{A}\left(Q_{{\mathbb{A}}}, S_{{\mathbb{A}}}\right)
        =\ &\left({\nabla}Q_{{\mathbb{A}}} + {\omega}\wedge S_{{\mathbb{A}}}, {\Phi}^\Lambda Q_{{\mathbb{A}}} - {\nabla} S_{{\mathbb{A}}}\right),
    \end{align}
we need to show that both components in (\ref{curvature with connection}) equal to $0$. 

Let $e_1, \cdots, e_n$ be an oriented local orthonormal frame of $E$. Then, we let
\begin{align}
    \nabla e_j =\ & \eta_{ij}\otimes e_i,\\
    \nabla^2 e_j =\ & R_{ij}\otimes e_i,\\
    \Phi e_j =\ &\varphi_{ij}e_i.
\end{align}
Here, $\eta_{ij}$ is a $1$-form, $R_{ij}$ is a $2$-form, and $\varphi_{ij}$ is a function. They satisfy
\begin{align}\label{skew-adjoint relations}
    \eta_{ij} = -\eta_{ji},\ R_{ij} = -R_{ji},\ \varphi_{ij} = -\varphi_{ji}
\end{align}
because $\mathbb{A}$ is Euclidean. Also, we have the structural identity \cite[Theorem 11.1]{loringtu2017differential}
\begin{align}\label{structural identity}
    R_{ij} =\ & d\eta_{ij} + \eta_{ik}\wedge\eta_{kj}
\end{align}
and the Bianchi identity \cite[Proposition 22.3]{loringtu2017differential}
\begin{align}\label{bianchi identity second one classical}
    d R_{ij} =\ & R_{ik}\wedge\eta_{kj} - \eta_{ik}\wedge R_{kj}.
\end{align}
Now, we see that
\begin{align}\label{first part first}
     & 2\left({\nabla}Q_{{\mathbb{A}}} + {\omega}\wedge S_{{\mathbb{A}}}\right) \nonumber \\
    =\ & \nabla\left(g(\nabla^2 e_i + \omega\wedge\Phi e_i, e_j)\otimes(e_i\wedge e_j)\right) + \omega\wedge g(\Phi\nabla e_i - \nabla\Phi e_i, e_j)\otimes(e_i\wedge e_j)\nonumber \\
    =\ & \nabla\left((R_{ji}+\varphi_{ji}\omega)\otimes(e_i\wedge e_j)\right) + \omega\wedge\left(\varphi_{jk}\eta_{ki} - d\varphi_{ji} - \varphi_{ki}\eta_{jk}\right)\otimes(e_i\wedge e_j)  \nonumber\\
    =\ & dR_{ji}\otimes(e_i\wedge e_j) + (R_{ji}\wedge\eta_{ki} + \varphi_{ji}\omega\wedge\eta_{ki})\otimes(e_k\wedge e_j)  +  (R_{ji}\wedge\eta_{kj} + \varphi_{ji}\omega\wedge\eta_{kj})\otimes(e_i\wedge e_k) \nonumber \\ 
    & + (\varphi_{jk}\omega\wedge\eta_{ki} - \varphi_{ki}\omega\wedge\eta_{jk})\otimes(e_i\wedge e_j)
\end{align}
Since $\eta_{ik} = -\eta_{ki}$, we have
\begin{align}\label{first part 2nd}
    & \varphi_{ji}\eta_{ki}\otimes(e_k\wedge e_j) + \varphi_{ji}\eta_{kj}\otimes(e_i\wedge e_k) + (\varphi_{jk}\eta_{ki} - \varphi_{ki}\eta_{jk})\otimes(e_i\wedge e_j)  \nonumber\\
    =\ & \varphi_{jk}\eta_{ik}\otimes(e_i\wedge e_j) + \varphi_{ki}\eta_{jk}\otimes(e_i\wedge e_j) + (\varphi_{jk}\eta_{ki} - \varphi_{ki}\eta_{jk})\otimes(e_i\wedge e_j) \nonumber\\
    =\ & 0. 
\end{align}
In addition, by (\ref{bianchi identity second one classical}) and $\eta_{ik} = -\eta_{ki}$, we have
\begin{align}\label{first part 3rd}
    & dR_{ji}\otimes(e_i\wedge e_j) + (R_{ji}\wedge\eta_{ki})\otimes(e_k\wedge e_j) + (R_{ji}\wedge\eta_{kj})\otimes(e_i\wedge e_k)  \nonumber\\
 =\ & dR_{ji}\otimes(e_i\wedge e_j) + (R_{jk}\wedge\eta_{ik})\otimes(e_i\wedge e_j) + (R_{ki}\wedge\eta_{jk})\otimes(e_i\wedge e_j)  \nonumber\\
 =\ & 0. 
\end{align}
By (\ref{first part first}), (\ref{first part 2nd}), and (\ref{first part 3rd}), we obtain
\begin{align}
    {\nabla}Q_{{\mathbb{A}}} + {\omega}\wedge S_{{\mathbb{A}}} = 0. 
\end{align}

Next, we find that 
\begin{align}\label{something we need to prove}
      & 2\left(\Phi^\Lambda Q_{\mathbb{A}} - \nabla S_{\mathbb{A}}\right)  \nonumber\\
    =\ & g(\nabla^2 e_i + \omega\wedge\Phi e_i, e_j)\otimes\Phi^\Lambda(e_i\wedge e_j) - \nabla\left(g(\Phi\nabla e_i - \nabla\Phi e_i, e_j)\otimes(e_i\wedge e_j)\right)  \nonumber\\
    =\ & (R_{ji} + \varphi_{ji}\omega)\otimes(\varphi_{ki}e_k\wedge e_j + \varphi_{k j} e_i\wedge e_k) - \nabla\left((\varphi_{jk}\eta_{ki} - d\varphi_{ji} - \varphi_{ki}\eta_{ji})\otimes(e_i\wedge e_j)\right) \nonumber \\
    =\ & (R_{ji} + \varphi_{ji}\omega)\otimes(\varphi_{ki}e_k\wedge e_j + \varphi_{k j} e_i\wedge e_k) - d(\varphi_{jk}\eta_{ki} - d\varphi_{ji} - \varphi_{ki}\eta_{jk})\otimes(e_i\wedge e_j) \nonumber\\
    & + (\varphi_{jk}\eta_{ki} - d\varphi_{ji} - \varphi_{ki}\eta_{jk})\wedge\eta_{ri}\otimes(e_r\wedge e_j) + (\varphi_{jk}\eta_{ki} - d\varphi_{ji} - \varphi_{ki}\eta_{jk})\wedge\eta_{rj}\otimes(e_i\wedge e_r). 
\end{align}
By $\varphi_{ik} = -\varphi_{ki}$, we have 
\begin{align}\label{remove 1}
    & \varphi_{ji}\omega\otimes(\varphi_{ki}e_k\wedge e_j + \varphi_{k j} e_i\wedge e_k) \nonumber\\
    =\ & \varphi_{ji}\varphi_{ki}\omega\otimes(e_k\wedge e_j) + \varphi_{ji}\varphi_{kj}\omega\otimes(e_i\wedge e_k)  \nonumber\\
    =\ & \varphi_{jk}\varphi_{ik}\omega\otimes(e_i\wedge e_j) + \varphi_{ki}\varphi_{jk}\omega\otimes(e_i\wedge e_j) \nonumber\\
    =\ & 0.
\end{align}
By $\eta_{ki} = -\eta_{ik}$, we have 
\begin{align}\label{remove 2}
    & d(\varphi_{jk}\eta_{ki} - d\varphi_{ji} - \varphi_{ki}\eta_{jk})\otimes(e_i\wedge e_j) + d\varphi_{ji}\wedge\left(\eta_{ri}\otimes(e_r\wedge e_j) +  \eta_{rj}\otimes(e_i\wedge e_r)\right) \nonumber \\
    =\ & d(\varphi_{jk}\eta_{ki} - \varphi_{ki}\eta_{jk})\otimes(e_i\wedge e_j) + (d\varphi_{jr}\wedge\eta_{ir})\otimes(e_i\wedge e_j) +  (d\varphi_{ri}\wedge\eta_{jr})\otimes(e_i\wedge e_j)   \nonumber \\ 
    =\ & \left(d\varphi_{jk}\wedge\eta_{ki} + \varphi_{jk}d\eta_{ki} - d\varphi_{ki}\wedge\eta_{jk} - \varphi_{ki}d\eta_{jk} + d\varphi_{jk}\wedge\eta_{ik} + d\varphi_{ki}\wedge\eta_{jk}\right)\otimes(e_i\wedge e_j)   \nonumber\\
    =\ & (\varphi_{jk}d\eta_{ki} - \varphi_{ki}d\eta_{jk})\otimes(e_i\wedge e_j). 
\end{align}
By (\ref{structural identity}), (\ref{something we need to prove}), (\ref{remove 1}), (\ref{remove 2}), $\varphi_{ki} = -\varphi_{ik}$, and $\eta_{ki} = -\eta_{ik}$, we have 
\begin{align}\label{final step}
      & 2\left(\Phi^\Lambda Q_{\mathbb{A}} - \nabla S_{\mathbb{A}}\right)  \nonumber\\
    =\ & R_{ji}\otimes(\varphi_{ki}e_k\wedge e_j + \varphi_{k j} e_i\wedge e_k) - (\varphi_{jk}d\eta_{ki} - \varphi_{ki}d\eta_{jk})\otimes(e_i\wedge e_j)\nonumber\\
    & + (\varphi_{jk}\eta_{ki} - \varphi_{ki}\eta_{jk})\wedge\eta_{ri}\otimes(e_r\wedge e_j) + (\varphi_{jk}\eta_{ki} - \varphi_{ki}\eta_{jk})\wedge\eta_{rj}\otimes(e_i\wedge e_r)  \nonumber\\
    =\ & \varphi_{ik}R_{jk}\otimes(e_i\wedge e_j) + \varphi_{jk}R_{ki}\otimes(e_i\wedge e_j) - (\varphi_{jk}d\eta_{ki} - \varphi_{ki}d\eta_{jk})\otimes(e_i\wedge e_j)\nonumber\\
    & + (\varphi_{jk}\eta_{kr} - \varphi_{kr}\eta_{jk})\wedge\eta_{ir}\otimes(e_i\wedge e_j) + (\varphi_{rk}\eta_{ki} - \varphi_{ki}\eta_{rk})\wedge\eta_{jr}\otimes(e_i\wedge e_j)  \nonumber\\
    =\ & (\varphi_{ik}R_{jk} + \varphi_{ki}d\eta_{jk} - \varphi_{ki}\eta_{rk}\wedge\eta_{jr})\otimes(e_i\wedge e_j)   \nonumber\\
   & +  (\varphi_{jk}R_{ki}-\varphi_{jk}d\eta_{ki}+\varphi_{jk}\eta_{kr}\wedge\eta_{ir})\otimes(e_i\wedge e_j)   \nonumber\\
  & +  (-\varphi_{kr}\eta_{jk}\wedge\eta_{ir} + \varphi_{rk}\eta_{ki}\wedge\eta_{jr})\otimes(e_i\wedge e_j)   \nonumber\\
  =\ & -\varphi_{kr}\eta_{jk}\wedge\eta_{ir}\otimes(e_i\wedge e_j) + \varphi_{rk}\eta_{ki}\wedge\eta_{jr}\otimes(e_i\wedge e_j)   \nonumber\\
  =\ & -\varphi_{kr}\eta_{jk}\wedge\eta_{ir}\otimes(e_i\wedge e_j) + \varphi_{kr}\eta_{ri}\wedge\eta_{jk}\otimes(e_i\wedge e_j)  \nonumber\\
  =\ & 0. 
\end{align}
Thus, by (\ref{final step}), we have 
\begin{align}
    \Phi^\Lambda Q_{\mathbb{A}} - \nabla S_{\mathbb{A}} = 0.
\end{align}
The proof of the proposition is complete. 
\end{proof}

\begin{remark}\normalfont 
    The proof of Proposition \ref{Bianchi identity another expression} relies on (\ref{skew-adjoint relations}). This is why we call (\ref{a acts on q and s}) the skew-adjoint Bianchi identity. See \cite[Proposition 2.5]{transgression_primitive_2025} for the Bianchi identity of a general mapping cone covariant derivative. 
\end{remark}

\section{Berezin integral of pairs}\label{section of berezin integral}
We continue with the settings of $E$ and $\mathbb{A}$ from Section \ref{section of bianchi identity}. In this section, we study the Berezin integral and extend it to pairs of bundle-valued differential forms. 

Based on \cite[(1.28)]{bgv}, \cite[(3.6)]{wittendeformationweipingzhang}, and \cite[(3.6)]{weipingzhangnewedition}, 
we begin with the Berezin integral 
\begin{align}
    \int^B: \Omega^*(M,\Lambda^*E)\to\Omega^*(M) 
\end{align}
associated with $g$. On a local chart $V\subseteq M$, we let
    $e_1, \cdots, e_n$
be an oriented orthonormal frame of $E$. Then, for any $\alpha\in\Omega^*(M,\Lambda^*E)$, we write
\begin{align}
    \alpha|_V = \sum_{r = 0}^n\hspace{+1mm}\sum_{1\leqslant i_1<\cdots<i_r\leqslant n}\alpha_{i_1\cdots i_r}\otimes (e_{i_1}\wedge\cdots\wedge e_{i_r}).
\end{align}
Here, $\alpha_{i_1\cdots i_r}\in\Omega^*(V)$.  
Then, the Berezin integral of $\alpha$ is given by 
\begin{align}\label{berezin integral local expression}
    \left.\left(\int^B \alpha\right)\right\vert_V = \alpha_{1\hspace{+0.25mm}2\hspace{+0.25mm}3\hspace{+0.25mm}\cdots\hspace{+0.25mm}n}.
\end{align}
The $\displaystyle\int^B \alpha$ is well-defined since the determinant of a special orthogonal matrix is $1$.  

We generalize the Berezin integral to the mapping cone situation: 
\begin{align}
   \int^B: \Omega^*(M,\Lambda^*E)\oplus\Omega^*(M,\Lambda^*E)&\to\Omega^*(M)\oplus\Omega^*(M) \nonumber\\
    (\alpha, \beta)&\mapsto \left(\int^B\alpha, \int^B\beta\right). 
\end{align}
It satisfies the following property:
\begin{proposition}
    For any $(\alpha,\beta)\in\Omega^*(M,\Lambda^*E)\oplus\Omega^*(M,\Lambda^*E)$, we have
    \begin{align}\label{d commutes with nabla in berezin integral}
        d^\omega\int^B (\alpha,\beta) = \int^B\mathbb{A}(\alpha,\beta). 
    \end{align}
\end{proposition}
\begin{proof}
    By \cite[Proposition 3.2]{wittendeformationweipingzhang}, we have 
    \begin{align}
        \int^B\mathbb{A}(\alpha,\beta) = (d\int^B\alpha + \omega\wedge\int^B\beta, \int^B\Phi^\Lambda\alpha - d\int^B\beta).
    \end{align}
    Let $e_1, \cdots e_n$ be an oriented local orthonormal frame of $E$. Since $\Phi$ is skew-adjoint with respect to $g$, we see that 
    \begin{align}
        g(\Phi(e_i), e_i) = 0\ \ (1\leqslant i\leqslant n),
    \end{align}
    and therefore
    \begin{align}
        \Phi^\Lambda(e_1\wedge\cdots\wedge e_n) = \sum_{j = 1}^n e_1\wedge\cdots\wedge \Phi(e_j)\wedge\cdots\wedge e_n = 0. 
    \end{align}
    Thus, we have 
    \begin{align}
        \int^B\mathbb{A}(\alpha,\beta) = \left(d\int^B\alpha + \omega\wedge\int^B\beta,\ - d\int^B\beta\right) = d^\omega\int^B (\alpha,\beta).
    \end{align}
    The proof is complete.
\end{proof}

\section{Thom form for the mapping cone complex}\label{section of mapping cone thom}
We continue with the settings of $E$ and $\mathbb{A}$ from Sections \ref{section of bianchi identity} and \ref{section of berezin integral}. In this section, for the de Rham mapping cone cochain complex, we give the Thom form
in the sense of \cite[Theorem 4.5]{Mathai-Quillen_Thom_form}. 

We follow the steps in \cite[Section 1.6]{bgv}, \cite[Chapter 3]{wittendeformationweipingzhang}, and \cite[Chapter 3]{weipingzhangnewedition}. Let 
\begin{align}
    \sigma: E\to M
\end{align}
be the projection from the total space of $E$ to the base space $M$. Then, we let 
$\widetilde{E}$ be the pullback bundle $\sigma^*E$ over $E$. Also, we let 
\begin{align}
    \widetilde{\omega} = \sigma^*\omega,\ 
    \widetilde{\nabla} = \sigma^*\nabla,\ 
    \widetilde{\Phi} = \sigma^*\Phi,\ 
    \widetilde{g} = \sigma^*g. 
\end{align}
Then, the $\widetilde{\Phi}\in\Omega^0(E,\edmp(\widetilde{E}))$ induces the associated derivation 
$\widetilde{\Phi}^\Lambda$ in the same manner as (\ref{derivation common pattern}).
Then, we have the de Rham mapping cone complex 
\begin{align}\label{mapping cone complex of the bundle}
    d^{\widetilde{\omega}}: \Omega^i(E)\oplus\Omega^{i-1}(E)&\to\Omega^{i+1}(E)\oplus\Omega^{i}(E)\nonumber\\
    (\alpha,\beta)&\mapsto (d\alpha + \widetilde{\omega}\wedge\beta, -d\beta),
\end{align}
 and the mapping cone covariant derivative
\begin{align}
    \widetilde{\mathbb{A}}:  \Omega^i(E, \Lambda^*\widetilde{E})\oplus\Omega^{i-1}(E, \Lambda^*\widetilde{E})&\to\Omega^{i+1}(E, \Lambda^*\widetilde{E})\oplus\Omega^{i}(E, \Lambda^*\widetilde{E}) \nonumber\\
    (\alpha, \beta) &\mapsto \left(\widetilde{\nabla}\alpha + \widetilde{\omega}\wedge\beta,\ \widetilde{\Phi}^\Lambda\alpha-\widetilde{\nabla}\beta\right).
\end{align}

Let $\mathbf{v}$ be the tautological section of $\widetilde{E}$. In fact, given an oriented local orthonormal frame $e_1, \cdots, e_n$ of $E$, we let $\widetilde{e}_i = \sigma^* e_i$ for $1\leqslant i\leqslant n$. Then, 
the tautological section $\mathbf{v}: E\to\widetilde{E}$
is defined by 
\begin{align}\label{coordinate expression of tautological section}
    \mathbf{v}(y_1 e_1 + \cdots + y_n e_n) = y_1\widetilde{e}_1 + \cdots + y_n\widetilde{e}_n
\end{align}
for all $y_1, \cdots, y_n\in\mathbb{R}$. Then, we have 
\begin{align}
    (\mathbf{v}, 0)\in\Omega^0(E, \widetilde{E})\oplus\Omega^{-1}(E, \widetilde{E})
\end{align}
and let 
\begin{align}
    |\mathbf{v}| = \sqrt{\widetilde{g}(\mathbf{v}, \mathbf{v})}. 
\end{align}
In addition, following (\ref{definition of q}) and (\ref{definition of s}), we have 
\begin{align}
    Q_{\widetilde{\mathbb{A}}} = \sum_{1\leqslant i<j\leqslant n}\widetilde{g}\left(\widetilde{\nabla}^2\widetilde{e}_i + \widetilde{\omega}\wedge\widetilde{\Phi}\widetilde{e}_i,\ \widetilde{e}_j\right) \otimes(\widetilde{e}_i\wedge \widetilde{e}_j)
    \in\Omega^2(E,\Lambda^2\widetilde{E})
\end{align}
and
\begin{align}
    S_{\widetilde{\mathbb{A}}} = \sum_{1\leqslant i< j\leqslant n}\widetilde{g}\left(\widetilde{\Phi}\widetilde{\nabla}\widetilde{e}_i - \widetilde{\nabla}\widetilde{\Phi}\widetilde{e}_i,\ \widetilde{e}_j\right)\otimes(\widetilde{e}_i\wedge \widetilde{e}_j) \in\Omega^1(E,\Lambda^2\widetilde{E}). 
\end{align}
Following \cite[Section 1.6 page 53]{bgv}, we let $\mathbf{v}\lrcorner$ be the contraction operator by $\mathbf{v}$: For any $\eta\in\Omega^i(E)$ and any smooth section $w: E\to\Lambda^*\widetilde{E}$, 
\begin{align}\label{original contraction}
    \mathbf{v}\lrcorner(\eta\otimes w)  = (-1)^i\eta\otimes(\mathbf{v}\lrcorner w)
\end{align}
defines the contraction by $\mathbf{v}$. Then, we extend the contraction to pairs: 
\begin{align}\label{extended contraction}
    \mathbf{v}\lrcorner: \Omega^i(E,\Lambda^*\widetilde{E})\oplus\Omega^{i-1}(E, \Lambda^*\widetilde{E}) &\to \Omega^i(E,\Lambda^*\widetilde{E})\oplus\Omega^{i-1}(E, \Lambda^*\widetilde{E}) \nonumber \\
  (\alpha, \beta) &\mapsto  (\mathbf{v}\lrcorner\alpha, -\mathbf{v}\lrcorner\beta).
\end{align}
The minus sign in $-\mathbf{v}\lrcorner\beta$ is for the compatibility with signs conventions and with the wedge product (\ref{define the wedge product for pairs}) in this pair situation. 

By the definition of the Berezin integral, we refine (\ref{d commutes with nabla in berezin integral}) to:
\begin{corollary}\label{swap d and connection in berezin integral}
For any $(\alpha, \beta)\in\Omega^*(E,\Lambda^*E)\oplus\Omega^*(E,\Lambda^*E)$, we have
    \begin{align}
        d^{\widetilde{\omega}}\int^B(\alpha, \beta) = \int^B \left(\widetilde{\mathbb{A}}+\mathbf{v}\lrcorner\right)(\alpha,\beta).
    \end{align}
\end{corollary}

Now, we let 
\begin{align}
    \mathcal{A} = \left(\dfrac{1}{2}|\mathbf{v}|^2, 0\right) + \widetilde{\mathbb{A}}(\mathbf{v},0) - \left(Q_{\widetilde{\mathbb{A}}}, S_{\widetilde{\mathbb{A}}}\right), 
\end{align}
and 
\begin{align}
    \mathcal{U} = (-1)^{n(n+1)/2}\left(\dfrac{1}{2\pi}\right)^{n/2}\int^B e^{-\mathcal{A}}. 
\end{align}
In the definition of $\mathcal{U}$, we use the Taylor expansion 
\begin{align}
    e^z = 1 + z + \dfrac{1}{2!}z^2 + \dfrac{1}{3!}z^3 + \cdots  
\end{align}
and the wedge product on $\Omega^*(E,\Lambda^*E)\oplus\Omega^*(E,\Lambda^*E)$ determined by 
\begin{align}\label{define the wedge product for pairs}
    & (\alpha\otimes x, \beta\otimes y)\wedge(\gamma\otimes u, \delta\otimes v) \nonumber\\
    =\ & \left((-1)^{|x||\gamma|}(\alpha\wedge\gamma)\otimes(x\wedge u),\hspace{+0.5mm}(-1)^{|y||\gamma|}(\beta\wedge\gamma)\otimes(y\wedge u) + (-1)^{|x|(|\delta|+1)}(\alpha\wedge\delta)\otimes (x\wedge v)\right).  \nonumber\\
    & (\text{Here, $|x|, |\gamma|$, etc., are the degrees.})
\end{align}
In other words, 
\begin{align}
    e^{-\mathcal{A}} = (1,0) + (-\mathcal{A}) + \dfrac{1}{2!}(-\mathcal{A})\wedge(-\mathcal{A}) + \dfrac{1}{3!}(-\mathcal{A})\wedge(-\mathcal{A})\wedge(-\mathcal{A})+\cdots.
\end{align}
This series stops because the rank of $E$ and the dimension of $M$ are both finite. 

Recall that $n$ is the rank of $E$. We now prove the main result Theorem \ref{main result}. 
\begin{proposition}
    The pair $\mathcal{U}\in\Omega^n(E)\oplus\Omega^{n-1}(E)$ is $d^{\widetilde{\omega}}$-closed. Also, it satisfies
    \begin{align}
        \int_{E/M}\mathcal{U} = (1,0).
    \end{align}
    Here, $\displaystyle\int_{E/M}$ means the integration along the fiber.
\end{proposition}
\begin{proof}
    First, the degree of $U$ is given by the rank of $E$ and (\ref{berezin integral local expression}). 
    
    Second, by Proposition \ref{Bianchi identity another expression}, we have 
    \begin{align}\label{c1}
        \widetilde{\mathbb{A}}\mathcal{A}
        =\ & \widetilde{\mathbb{A}}\left(\dfrac{1}{2}|\mathbf{v}|^2, 0\right) + \left((\widetilde{\nabla}^2 + \widetilde{\omega}\wedge\widetilde{\Phi})\mathbf{v}, (\widetilde{\Phi}\widetilde{\nabla} - \widetilde{\nabla}\widetilde{\Phi})\mathbf{v}\right) - \widetilde{\mathbb{A}}\left(Q_{\widetilde{\mathbb{A}}}, S_{\widetilde{\mathbb{A}}}\right) \nonumber\\
        =\ & (-\mathbf{v}\lrcorner\widetilde{\nabla}\mathbf{v}, 0) + \left((\widetilde{\nabla}^2 + \widetilde{\omega}\wedge\widetilde{\Phi})\mathbf{v}, (\widetilde{\Phi}\widetilde{\nabla} - \widetilde{\nabla}\widetilde{\Phi})\mathbf{v}\right)
    \end{align}
    and then 
    \begin{align}\label{c2}
       \mathbf{v}\lrcorner\mathcal{A} =\ & \mathbf{v}\lrcorner\left(\dfrac{1}{2}|\mathbf{v}|^2, 0\right) + (\mathbf{v}\lrcorner\widetilde{\nabla}\mathbf{v}, -\mathbf{v}\lrcorner\widetilde{\Phi}\mathbf{v}) - \left(\mathbf{v}\lrcorner Q_{\widetilde{\mathbb{A}}}, -\mathbf{v}\lrcorner S_{\widetilde{\mathbb{A}}}\right)  \nonumber\\
       =\ & (\mathbf{v}\lrcorner\widetilde{\nabla}\mathbf{v}, -\mathbf{v}\lrcorner\widetilde{\Phi}\mathbf{v}) - \left(\mathbf{v}\lrcorner Q_{\widetilde{\mathbb{A}}}, -\mathbf{v}\lrcorner S_{\widetilde{\mathbb{A}}}\right).
    \end{align}
    Because $\widetilde{\Phi}$ is skew-adjoint with respect to $\widetilde{g}$, we notice that 
    \begin{align}\label{c3}
        -\mathbf{v}\lrcorner\widetilde{\Phi}\mathbf{v} = -\widetilde{g}(\widetilde{\Phi}\mathbf{v}, \mathbf{v}) = 0. 
    \end{align}
    Also, we notice that 
    \begin{align}\label{c4}
        \mathbf{v}\lrcorner Q_{\widetilde{\mathbb{A}}} = (\widetilde{\nabla}^2 + \widetilde{\omega}\wedge\widetilde{\Phi})\mathbf{v}
    \end{align}
    and that 
    \begin{align}\label{c5}
        -\mathbf{v}\lrcorner S_{\widetilde{\mathbb{A}}} = (\widetilde{\Phi}\widetilde{\nabla} - \widetilde{\nabla}\widetilde{\Phi})\mathbf{v}. 
    \end{align}
By (\ref{c1}), (\ref{c2}), (\ref{c3}), (\ref{c4}), and (\ref{c5}), we have 
\begin{align}
    (\widetilde{\mathbb{A}} + \mathbf{v}\lrcorner)\mathcal{A} = (0, 0).
\end{align}
By (\ref{original contraction}) and (\ref{extended contraction}) and checking the degree of $\mathcal{A}$, we have
\begin{align}
   \hspace{-2.61mm} & \mathbf{v}\lrcorner(\mathcal{A}\wedge\cdots\wedge\mathcal{A}) \nonumber\\
   \hspace{-2.61mm} =\ & (\mathbf{v}\lrcorner\mathcal{A})\wedge\mathcal{A}\wedge\cdots\wedge\mathcal{A} + \mathcal{A}\wedge(\mathbf{v}\lrcorner\mathcal{A})\wedge\cdots\wedge\mathcal{A} + \cdots + \mathcal{A}\wedge\mathcal{A}\wedge\cdots\wedge(\mathbf{v}\lrcorner\mathcal{A}).
\end{align}
Then, we find $\mathcal{U}$ is $d^{\widetilde{\omega}}$-closed:   
\begin{align}
    d^{\widetilde{\omega}}\mathcal{U} = (-1)^{n(n+1)/2}\left(\dfrac{1}{2\pi}\right)^{n/2}\int^B (\widetilde{\mathbb{A}} + \mathbf{v}\lrcorner) e^{-\mathcal{A}} = 0. 
\end{align}

Third, by the definition \cite[(4.4)]{bott2013differential} of the integration along the fiber, we use the local expression (\ref{coordinate expression of tautological section}) and obtain 
\begin{align}
     & \int_{E/M}\mathcal{U} \nonumber\\
    =\ & \left((-1)^{n(n+1)/2}\left(\dfrac{1}{2\pi}\right)^{n/2}\int_{\mathbb{R}^n}e^{-\frac{1}{2}(y_1^2 + \cdots + y_n^2)}\int^B \dfrac{1}{n!}(1-dy_1\otimes\widetilde{e}_1 - \cdots - dy_n\otimes\widetilde{e}_n)^n,\ 0\right).
\end{align}
Here, the power $n$ on $(1-dy_1\otimes\widetilde{e}_1 - \cdots - dy_n\otimes\widetilde{e}_n)^n$ means wedge $n$ times. Since 
\begin{align}
     (dy_i\otimes \widetilde{e}_i)\wedge(dy_j\otimes \widetilde{e}_j) = -(dy_i\wedge dy_j)\otimes(\widetilde{e}_i\wedge \widetilde{e}_j) = (dy_j\otimes \widetilde{e}_j)\wedge (dy_i\otimes \widetilde{e}_i),
\end{align}
then as in \cite[Proposition 3.1]{wittendeformationweipingzhang}, we see that 
\begin{align}
    &  (-1)^{n(n+1)/2}\left(\dfrac{1}{2\pi}\right)^{n/2}\int_{\mathbb{R}^n}e^{-\frac{1}{2}(y_1^2 + \cdots + y_n^2)}\int^B \dfrac{1}{n!}(1-dy_1\otimes\widetilde{e}_1 - \cdots - dy_n\otimes\widetilde{e}_n)^n   \nonumber\\
    =\ & \left(\dfrac{1}{2\pi}\right)^{n/2}\int_{\mathbb{R}^n}e^{-\frac{1}{2}(y_1^2 + \cdots + y_n^2)}dy_1\wedge\cdots\wedge dy_n   \nonumber\\
    =\ & 1. 
\end{align}
The proof of Theorem \ref{main result} is complete.
\end{proof}
\begin{remark}\normalfont
    As we see, it is not very hard to obtain the normal distribution and the Gaussian integral. The curved terms does not contribute to the integration along the fiber. In fact, in both Mathai-Quillen's original work and our current construction, the hardest task is to put on the curved terms so that $\mathcal{U}$ is closed. 
\end{remark}
\begin{remark}\normalfont
In \cite[Setting 3.1]{morrison2024cohomology}, the form $\widetilde{\omega}$ is replaced by $\widetilde{\omega} + d\mu$. This $\mu$ is a $1$-form on the total space of $E$. Actually, the map
\begin{align}
    \varrho: \Omega^i(E)\oplus\Omega^{i-1}(E)&\to\Omega^{i}(E)\oplus\Omega^{i-1}(E)  \nonumber\\
    (\alpha, \beta)&\mapsto (\alpha+\mu\wedge\beta, \beta)
\end{align}
satisfies 
\begin{align}
    d^{\widetilde{\omega}}\varrho(\alpha,\beta) = \varrho\hspace{+0.5mm} d^{\widetilde{\omega}+d\mu}(\alpha,\beta).
\end{align}
This means that adding or subtracting $d\mu$ does not affect the cohomology computed by (\ref{mapping cone complex of the bundle}). 
Thus, we will only use $\widetilde{\omega}$. 
\end{remark}

\section{Transgression of the Thom form}\label{section of transgression}
We continue with the settings of $E$, $\mathbb{A}$, $\widetilde{E}$, and $\widetilde{\mathbb{A}}$ from Sections \ref{section of bianchi identity}, \ref{section of berezin integral}, and \ref{section of mapping cone thom}. In this section, we prove the  transgression formula for the Thom form $\mathcal{U}$. The proof follows the pattern of \cite[Proposition 1.54]{bgv}. 

We fix the form $\omega$, the metric $g$, and the tautological section $\mathbf{v}$. Then, we let $\nabla_t\ (t\in\mathbb{R})$ be a smooth family of Euclidean connections with respect to $g$, and let 
\begin{align}
    \Phi_t\in\Omega^0(M,\edmp(E))\ (t\in\mathbb{R})
\end{align}
be a smooth family of skew-adjoint endomorphisms with respect to $g$. Consequently, we obtain $\Phi_t^\Lambda$, $\mathbb{A}_t$, $\widetilde{\nabla}_t$, $\widetilde{\Phi}_t^\Lambda$, $\widetilde{\Phi}_t$,  $\widetilde{\mathbb{A}}_t$, and then $\mathcal{U}_t$. 

Besides using (\ref{definition of q}) and (\ref{definition of s}) to get $Q_{{\mathbb{A}}_t}$ and $S_{{\mathbb{A}}_t}$, we also define 
\begin{align}
    Y_{\mathbb{A}_t} = \sum_{1\leqslant i<j\leqslant n}g\left(\dfrac{d\nabla_t}{dt}e_i, e_j\right)\otimes(e_i\wedge e_j)
\end{align}
and 
\begin{align}
    Z_{\mathbb{A}_t} = \sum_{1\leqslant i<j\leqslant n}g\left(\dfrac{d\Phi_t}{dt}e_i, e_j\right)\otimes(e_i\wedge e_j).
\end{align}
They are independent of the chosen oriented local orthonormal frame $e_1, \cdots, e_n$. 
\begin{proposition}
    For $t\in\mathbb{R}$, we have 
    \begin{align}\label{connection on connection identified forms}
        \mathbb{A}_t(Y_{\mathbb{A}_t}, Z_{\mathbb{A}_t}) = \left(\dfrac{dQ_{{\mathbb{A}}_t}}{dt}, \dfrac{dS_{{\mathbb{A}}_t}}{dt}\right).
    \end{align}
\end{proposition}
\begin{proof}
    We need to show that 
    \begin{align}
        {\nabla}_t Y_{{\mathbb{A}_t}} + {\omega}\wedge Z_{{\mathbb{A}_t}} & = \dfrac{dQ_{{\mathbb{A}}_t}}{dt},  \label{transgression first}\\
      \text{and\ \ }  {\Phi}_t^\Lambda Y_{{\mathbb{A}_t}} - {\nabla}_t Z_{{\mathbb{A}_t}} & = \dfrac{dS_{{\mathbb{A}}_t}}{dt}.  \label{transgression second}
    \end{align}
    For $t\in\mathbb{R}$, we let 
    \begin{align}
        \nabla_t e_j = \eta_{ij}\otimes e_i,
    \end{align}
    where $\eta_{ij}$ is smooth in terms of $t$. We omit $t$ in $\eta_{ij}$ for convenience. 
    Then, we have 
    \begin{align}
        & 2{\nabla}_t Y_{{\mathbb{A}_t}}  \nonumber\\
        =\ & \sum_{i, j = 1}^n d\left(g\left(\dfrac{d\nabla_t}{dt}e_i, e_j\right)\right)\otimes(e_i\wedge e_j) - \sum_{i, j = 1}^n g\left(\dfrac{d\nabla_t}{dt}e_i, e_j\right)\wedge\nabla_t (e_i\wedge e_j)   \nonumber\\
        =\ & \sum_{i, j = 1}^n \left(\dfrac{d}{dt}d\eta_{ji}\right)\otimes(e_i\wedge e_j) - \sum_{i,j,k = 1}^n\left(\dfrac{d\eta_{ji}}{dt}\wedge\eta_{ki}\right)\otimes(e_k\wedge e_j) + \sum_{i,j,k = 1}^n\left(\dfrac{d\eta_{ji}}{dt}\wedge\eta_{kj}\right)\otimes(e_i\wedge e_k)   \nonumber\\
        =\ & \sum_{i, j = 1}^n \dfrac{d}{dt}(d\eta_{ji} + \eta_{jk}\wedge\eta_{ki})\otimes(e_i\wedge e_j)  \text{\ \ \ \ (This line is guaranteed by $\eta_{ik} = -\eta_{ki}$.)} \nonumber\\
        =\ & 2\dfrac{dQ_{{\mathbb{A}}_t}}{dt} - 2{\omega}\wedge Z_{{\mathbb{A}_t}}. 
    \end{align}
    The proof of (\ref{transgression first}) is complete. Next, we let 
    \begin{align}
        \Phi e_j = \varphi_{ij}e_i.
    \end{align}
    Here, $\varphi_{ij}$ is smooth in terms of $t$, and the notation $t$ is also omitted. We find that 
    \begin{align}\label{1-1}
          & 2\dfrac{dS_{\mathbb{A}_t}}{dt}   \nonumber\\
         =\ & -\sum_{i, j = 1}^n \dfrac{d}{dt}(d\varphi_{ji})\otimes(e_i\wedge e_j) + \sum_{i, j, k = 1}^n \dfrac{d}{dt}\left(\varphi_{jk}\eta_{ki} - \varphi_{ki}\eta_{jk}\right)\otimes(e_i\wedge e_j).
    \end{align}
    Also, we have
    \begin{align}\label{1-2}
          & 2{\Phi}_t^\Lambda Y_{{\mathbb{A}_t}}  \nonumber\\
        =\ & \Phi^\Lambda_t\left(\sum_{1\leqslant i, j\leqslant n}\dfrac{d\eta_{ji}}{dt}\otimes(e_i\wedge e_j)\right)    \nonumber\\
        =\ & \sum_{i, j, k = 1}^n \left(\varphi_{ik}\dfrac{d\eta_{jk}}{dt} + \varphi_{jk}\dfrac{d\eta_{ki}}{dt}\right)\otimes(e_i\wedge e_j)
    \end{align}
    and 
    \begin{align}\label{2-1}
         & 2\nabla_t Z_{\mathbb{A}_t}  \nonumber\\
        =\ & \nabla_t\left(\sum_{i, j = 1}^n \dfrac{d\varphi_{ji}}{dt}(e_i\wedge e_j)\right)     \nonumber\\
        =\ & \sum_{i, j = 1}^n \dfrac{d}{dt}(d\varphi_{ji})\otimes(e_i\wedge e_j) + \sum_{i,j,k = 1}^n\left(\dfrac{d\varphi_{jk}}{dt}\eta_{ik} + \dfrac{d\varphi_{ki}}{dt}\eta_{jk}\right)\otimes (e_i\wedge e_j).
    \end{align}
    Using (\ref{1-1}), (\ref{1-2}), (\ref{2-1}), $\varphi_{ik} = -\varphi_{ki}$, and $\eta_{ik} = -\eta_{ki}$, we obtain (\ref{transgression second}). 
\end{proof}

Now, like \cite[Proposition 1.54]{bgv}, we prove the transgression formula: 
\begin{proposition}\label{proposition of type 1 transgression}
    For the family $\mathbb{A}_t$, we have
    \begin{align}\label{transgression type 1}
        \dfrac{d\hspace{+0.5mm}\mathcal{U}_t}{dt} = d^{\widetilde{\omega}}\left((-1)^{n(n+1)/2}\left(\dfrac{1}{2\pi}\right)^{n/2}\int^B \left(Y_{\widetilde{\mathbb{A}}_t}, Z_{\widetilde{\mathbb{A}}_t}\right)\wedge e^{-\mathcal{A}_t}\right).
    \end{align}
\end{proposition}
\begin{proof}
For the covariant derivative $\widetilde{\mathbb{A}}_t$ on $\widetilde{E}$, we have 
\begin{align}
    Y_{\widetilde{\mathbb{A}}_t} = \sum_{1\leqslant i<j\leqslant n}\widetilde{g}\left(\dfrac{d\widetilde{\nabla}_t}{dt}\widetilde{e}_i, \widetilde{e}_j\right)\otimes(\widetilde{e}_i\wedge \widetilde{e}_j)
\end{align}
and 
\begin{align}
    Z_{\widetilde{\mathbb{A}}_t} = \sum_{1\leqslant i<j\leqslant n}\widetilde{g}\left(\dfrac{d\widetilde{\Phi}_t}{dt}\widetilde{e}_i, \widetilde{e}_j\right)\otimes(\widetilde{e}_i\wedge \widetilde{e}_j). 
\end{align}
    Using the expression 
    \begin{align}
    \mathcal{A}_t = \left(\dfrac{1}{2}|\mathbf{v}|^2, 0\right) + \widetilde{\mathbb{A}}_t(\mathbf{v},0) - \left(Q_{\widetilde{\mathbb{A}}_t}, S_{\widetilde{\mathbb{A}}_t}\right)
\end{align}
together with (\ref{connection on connection identified forms}), we see that 
\begin{align}\label{key point almost like the bianchi identity}
    \dfrac{d\mathcal{A}_t}{dt} =\ & (0,0) -\mathbf{v}\lrcorner\left(Y_{\widetilde{\mathbb{A}}_t}, Z_{\widetilde{\mathbb{A}}_t}\right) -\widetilde{\mathbb{A}}_t\left(Y_{\widetilde{\mathbb{A}}_t}, Z_{\widetilde{\mathbb{A}}_t}\right)   \nonumber\\
    =\ & -(\widetilde{\mathbb{A}}_t+\mathbf{v}\lrcorner)\left(Y_{\widetilde{\mathbb{A}}_t}, Z_{\widetilde{\mathbb{A}}_t}\right).
\end{align}
By (\ref{key point almost like the bianchi identity}) and 
\begin{align}
    (\widetilde{\mathbb{A}}_t+\mathbf{v}\lrcorner)\mathcal{A}_t = (0,0),
\end{align}
we find  
\begin{align}
    \dfrac{d\hspace{+0.5mm}\mathcal{U}_t}{dt} 
    =\ & (-1)^{n(n+1)/2}\left(\dfrac{1}{2\pi}\right)^{n/2}\int^B -\dfrac{d\mathcal{A}_t}{dt}\wedge e^{-\mathcal{A}_t}   \nonumber\\
    =\ & (-1)^{n(n+1)/2}\left(\dfrac{1}{2\pi}\right)^{n/2}\int^B (\widetilde{\mathbb{A}}_t+\mathbf{v}\lrcorner)\left(\left(Y_{\widetilde{\mathbb{A}}_t}, Z_{\widetilde{\mathbb{A}}_t}\right)\wedge e^{-\mathcal{A}_t}\right)  \nonumber\\
    =\ & d^{\widetilde{\omega}}\left((-1)^{n(n+1)/2}\left(\dfrac{1}{2\pi}\right)^{n/2}\int^B \left(Y_{\widetilde{\mathbb{A}}_t}, Z_{\widetilde{\mathbb{A}}_t}\right)\wedge e^{-\mathcal{A}_t}\right).
\end{align}
The last equality is by Corollary \ref{swap d and connection in berezin integral}. The proof of Theorem \ref{main result 2} is complete. 
\end{proof}

We end this paper by presenting some inspirations related to this study: 

There should also be another type of transgression (cf. \cite[Proposition 1.53]{bgv} and \cite[Proposition 3.5]{wittendeformationweipingzhang}). In that type, $\mathbf{v}$ is rescaled to $t\mathbf{v}$, while $g$, $\omega$, $\nabla$, and $\Phi$ are unchanged. The readers may expect, when $\dim M$ is even, that type of transgression together with a nondegenerate vector field might bring a mapping cone analogue of the Gauss-Bonnet-Chern theorem. However, this is impossible. For example, when $M$ is closed and $\omega$ is a symplectic form whose de Rham cohomology class is integral, the de Rham mapping cone cochain complex computes the cohomology of a circle bundle $C$ over $M$. The $\dim C$ is odd, and
\begin{align}
    (0,1)\in\Omega^1(M)\oplus\Omega^0(M)
\end{align}
corresponds to a nonvanishing $1$-form $\theta\in\Omega^1(C)$. Intuitively speaking, for any nondegenerate vector field $X$ on $M$, when we use $X$ to pull back $\mathcal{U}$, the zeros of $X$ behave like isolated circles on $C$ instead of like isolated points on $M$. This means that the analysis \cite[(3.26)-(3.28)]{wittendeformationweipingzhang} around isolated zeros cannot be directly transferred to $X^*\mathcal{U}$ in this paper. 

Such a feature also explains why in spirit, the mapping cone Morse theory \cite{tangtsengclausenmappingcone, tangtsengclausensymplecticwitten, 2026_instanton_construction_mapping_cone} is kind of like Morse-Bott theory, no matter whether $\omega$ is symplectic or not. In addition, as supported by the transgression \ref{transgression type 1}, we can let $\Phi = 0$ if we need a simplified representative of the mapping cone Thom class. Thus, the main difficulty is the analysis on $\omega$ when studying the analytic side of the mapping cone Morse theory. See \cite{tangtsengclausensymplecticwitten, 2026_instanton_construction_mapping_cone} for more details.

Last but not least, when $M$ is a smooth closed oriented manifold, the Euler characteristic of the de Rham mapping cone cochain complex is determined only by the degree of $\omega$. It is equal to $0$ when the degree of $\omega$ is an even number. See \cite[Proposition 1.6]{brown1982complete} and \cite[Theorem 1.3]{tangtsengclausenmappingcone} for precise formulas. 
\bibliographystyle{abbrv}
\bibliography{mybib.bib}

@book{wittendeformationweipingzhang,
  title={Lectures on Chern-Weil Theory and Witten Deformations},
  author={Zhang, W.},
  isbn={9789812386588},
  year={2001},
  publisher={World Scientific}
}

@book{bott2013differential,
  title={Differential Forms in Algebraic Topology},
  author={Bott, R. and Tu, L. W.},
  isbn={9781475739510},
  lccn={81009172},
  year={1982},
  publisher={Springer}
}

@book{weipingzhangnewedition,
author = {Weiping Zhang and Huitao Feng},
 title = {Geometry and Analysis on Manifolds},
year = {2022},
edition = {{C}hinese},
publisher = {Higher Education Press}
}

@article{tanaka_tseng_2018,
    author = {Hiro Lee Tanaka and Li-Sheng Tseng},
    title = {Odd sphere bundles, symplectic manifolds, and their intersection theory},
    journal = {Cambridge Journal of Mathematics},
    volume = {6},
    number = {3},
    year = {2018},
    pages = {213--266},
    doi = {10.4310/CJM.2018.v6.n3.a1},
    issn = {2168-0930},
    publisher = {International Press of Boston}
}

@article{tty1st,
author = {Li-Sheng Tseng and Shing-Tung Yau},
title = {{Cohomology and Hodge theory on symplectic manifolds: I}},
volume = {91},
journal = {Journal of Differential Geometry},
number = {3},
publisher = {Lehigh University},
pages = {383--416},
year = {2012},
doi = {10.4310/jdg/1349292670},
URL = {https://doi.org/10.4310/jdg/1349292670}
}

@article{tty2nd,
author = {Li-Sheng Tseng and Shing-Tung Yau},
title = {{Cohomology and Hodge theory on symplectic manifolds: II}},
volume = {91},
journal = {Journal of Differential Geometry},
number = {3},
publisher = {Lehigh University},
pages = {417--443},
year = {2012},
doi = {10.4310/jdg/1349292671},
URL = {https://doi.org/10.4310/jdg/1349292671}
}

@article{tty3rd,
author = {Chung-Jun Tsai and Li-Sheng Tseng and Shing-Tung Yau},
title = {{Cohomology and Hodge theory on symplectic manifolds: III}},
volume = {103},
journal = {Journal of Differential Geometry},
number = {1},
publisher = {Lehigh University},
pages = {83--143},
year = {2016},
doi = {10.4310/jdg/1460463564},
URL = {https://doi.org/10.4310/jdg/1460463564}
}

@article{tangtsengclausensymplecticwitten,
      title={Symplectic {M}orse theory and {W}itten deformation}, 
      author={David Clausen and Xiang Tang and Li-Sheng Tseng},
      journal = {Mathematische Zeitschrift}, 
      volume={312},
  number={3},
  pages={94},
  year={2026}
}

@book{bgv,
  title={Heat Kernels and Dirac Operators},
  author={Berline, N. and Getzler, E. and Vergne, M.},
  year={2004},
  publisher={Springer}
}

@article{tseng_and_zhou_2025mapping_yang_mills,
  title={Mapping Cone Connections and their {Yang--Mills functional}},
  author={Tseng, Li-Sheng and Zhou, Jiawei},
  journal={Communications in Mathematical Physics},
  volume={406},
  number={7},
  pages={156},
  year={2025},
  publisher={Springer}
}

@article{tseng_and_zhou_symplectic_flat_connection_and_twisted_primitive2022,
  title={Symplectic flatness and twisted primitive cohomology},
  author={Tseng, Li-Sheng and Zhou, Jiawei},
  journal={The Journal of Geometric Analysis},
  volume={32},
  number={11},
  pages={282},
  year={2022},
  publisher={Springer}
}

@misc{symplectic_semi_char_2025,
      title={Symplectic semi-characteristics}, 
      author={Hao Zhuang},
      year={2025},
      eprint={2505.14496v1},
      archivePrefix={arXiv},
      primaryClass={math.DG},
      url={https://arxiv.org/abs/2505.14496}, 
      note = {ar{X}iv:2505.14496}
}

@misc{1_filtered_semi_char_2025,
      title={A vanishing property about the $1$-filtered cohomology groups of (4n+2)-dimensional closed symplectic manifolds}, 
      author={Hao Zhuang},
      year={2025},
      eprint={2510.10630v1},
      archivePrefix={arXiv},
      primaryClass={math.SG},
      url={https://arxiv.org/abs/2510.10630}, 
      note = {ar{X}iv:2510.10630}
}

@article{tangtsengclausenmappingcone,
      title={Mapping cone and {M}orse theory}, 
      author={David Clausen and Xiang Tang and Li-Sheng Tseng},
      Volume = {14}, 
      Number = {1}, 
      pages = {117--154}, 
      year = {2026},
      journal = {Cambridge Journal of Mathematics}
}

@misc{tseng_zhou_symplectic_flat_functional_characteristic_classes2022,
  title={Symplectically flat connections and their functionals},
  author={Tseng, Li-Sheng and Zhou, Jiawei},
  year={2022},
  note = {ar{X}iv:2210.03032v5}
}

@misc{transgression_primitive_2025,
      title={Transgression in the primitive cohomology}, 
      author={Hao Zhuang},
      year={2025},
      eprint={2512.24920},
      archivePrefix={arXiv},
      primaryClass={math.DG},
      url={https://arxiv.org/abs/2512.24920}, 
      note = {arXiv:2512.24920}
}

@article{Mathai-Quillen_Thom_form,
title = {Superconnections, {T}hom classes, and equivariant differential forms},
journal = {Topology},
volume = {25},
number = {1},
pages = {85-110},
year = {1986},
issn = {0040-9383},
doi = {https://doi.org/10.1016/0040-9383(86)90007-8},
url = {https://www.sciencedirect.com/science/article/pii/0040938386900078},
author = {Varghese Mathai and Daniel Quillen}
}

@book{bourbaki_algebra_1,
  title={Algebra I: Chapters 1-3},
  author={Bourbaki, N.},
  isbn={9783540642435},
  lccn={88031211},
  url={https://books.google.com/books?id=STS9aZ6F204C},
  year={1989},
  publisher={Springer}
}

@article{brown1982complete,
  title={Complete {E}uler characteristics and fixed-point theory},
  author={Brown, Kenneth S},
  journal={Journal of Pure and Applied Algebra},
  volume={24},
  number={2},
  pages={103--121},
  year={1982}
}

@phdthesis{morrison2024cohomology,
  title={Cohomology of Symplectic Manifolds},
  author={Morrison, Daniel},
  year={2024},
  school={University of California, Irvine}
}

@book{loringtu2017differential,
  title={Differential Geometry: Connections, Curvature, and Characteristic Classes},
  author={Tu, L. W.},
  isbn={9783319550848},
  url={https://books.google.com/books?id=bmsmDwAAQBAJ},
  year={2017},
  publisher={Springer}
}

@misc{2026_instanton_construction_mapping_cone,
      title={Instanton construction of the mapping cone {Thom-Smale} complex}, 
      author={Hao Zhuang},
      year={2026},
      eprint={2603.08404},
      archivePrefix={arXiv},
      primaryClass={math.DG},
      url={https://arxiv.org/abs/2603.08404}, 
      note = {arXiv:2603.08404}
}
\end{document}